\documentclass{article}
\usepackage{graphicx} % Required for inserting images

\usepackage{placeins}
\usepackage{amsthm,amsfonts,amssymb,amsmath,epsf, verbatim}

\def\kl{{\rm KL}}

\newtheorem{theorem}{Theorem}
\newtheorem{lemma}[theorem]{Lemma}
\newtheorem{corollary}[theorem]{Corollary}
\newtheorem{proposition}[theorem]{Proposition}

\newtheorem{conjecture}[theorem]{Conjecture}
\newtheorem{question}[theorem]{Question}
\newtheorem{definition}{Definition}

\newcommand{\dest}{{D_{EST}(P \parallel Q)}}

\title{Kullback–Leibler divergence and primitive non-deficient numbers}
\author{Joshua Zelinsky\footnote{Department of Mathematics, Hopkins School, New Haven, CT. USA jzelinsky@hopkins.edu}, Kyle Zhang\footnote{Hopkins School, New Haven, CT. USA, kzhang25@students.hopkins.edu}}
\date{}

\begin{document}

\maketitle

\begin{abstract} Let $H(n) = \prod_{p|n}\frac{p}{p-1}$ where $p$ ranges over the primes which divide $n$. It is well known that if $n$ is a primitive non-deficient number, then  $H(n) > 2$. We examine inequalities of the form $H(n)> 2 + f(n)$ for various functions $f(n)$ where $n$ is assumed to be primitive non-deficient and connect these inequalities to applying the Kullback-Leibler divergence to  different probability distributions on the set of divisors of $n$.
    
\end{abstract}

\section{Introduction}

Let $P$ and $Q$ be discrete probability distributions on the same sample space $\mathcal{X}$. Then the Kullback-Leibler divergence is defined as

\begin{equation}
    D_\kl(P \parallel Q) = \sum_{ x \in \mathcal{X} } P(x)\ \log\left(\frac{\ P(x)\ }{ Q(x) }\right)\ \label{def of KL divergence}.
\end{equation}

The Kullback-Leibler divergence is one way to measure the dissimilarity between two probability distributions on the same sample space. Less similar distributions have greater $D_\kl(P \parallel Q)$.    Note that $D_\kl$ is not a metric on the probability distributions since it fails both the needed symmetry and the triangle inequality. However, $D_\kl(P \parallel Q)$ is non-negative. Moreover,  $D_\kl(P \parallel Q) =0$ if and only if $P$ and $Q$ are the same probability distribution. 

The main idea  of this article is that certain numbers have two natural probability distributions on their divisors. We can then apply the positivity of the Kullback-Leibler divergence to those distributions to derive various inequalities.  Let $\phi(n)$ be the Euler totient function which counts the number of positive integers which are less than or equal to $n$ and relatively prime to $n$. Then A classical identity is the statement that

\begin{equation} \sum_{d|n} \phi(d) = n,
\end{equation}
or equivalently 

\begin{equation} \sum_{d|n} \frac{\phi(d)}{d} = 1.
\end{equation}

Thus, all integers $n$ have a probability distribution  on $\mathcal{X} = \{d>0:d \mid n\}$ which arises naturally from applying $\phi(d)$ to the divisors.

Let $\sigma(n)$ be the sum of the divisors of $n$. For example, $\sigma(10) =1 +2+5+10=18$. Recall that a number $n$ is said to be perfect if $\sigma(n)=2n$, said to be deficient if $\sigma(n) <2n$, and said to be abundant if $\sigma(n) > 2n$. We will also refer to $n$ as non-deficient if $n$ is perfect or abundant. The first few perfect numbers are 6, 28, 496, and 8128. One of the oldest unsolved problems is whether there are any odd numbers which are perfect. Note that the statement ``$n$ is perfect'' is the same as saying that

\begin{equation} \sum_{d|n} \frac{d}{2n} = 1. \label{probability distribution phrasing for n perfect}
\end{equation}

Equation (\ref{probability distribution phrasing for n perfect}) is a probability distribution on the sample space $\mathcal{X}$ of positive divisors of $n$.  Thus, if $n$ is perfect, there are two natural  probability distributions on its divisors. If $P$ is the probability distribution from $n$ being perfect, and $Q$ is the probability distribution arising from the Euler totient function, then the Kullback-Liebler divergence of these two distributions then is

\begin{equation}
D_\kl(P \parallel Q) = \sum_{d|n} \frac{d}{2n} \log \frac{d}{2\phi(d)} = \frac{1}{2n}\sum_{d|n} d\log \frac{d}{2\phi(d)}    \label{KL for divergence}.
\end{equation} Motivated by this observation, we will define the function $\kl(n)$ by  \begin{equation}
\kl (n) = \sum_{d|n} d\log \frac{d}{2\phi(d)}  \label{KL(n) def}. \end{equation}
 The function  $\kl (n)$ will be the main focus of the remainder of this paper. A skeptical reader may ask if this is simply an arbitrary choice to combine these two probability distributions that will not yield any substantial insight. However, there are multiple connections between $\phi(n)$ and the sum of the divisors of $n$. The first connection is through Pillai's function. Pillai's function,  $Pi(n),$ is defined by 

$$Pi(n)= \sum_{k=1}^n \mathrm{gcd}(k,n). $$ Pillai \cite{Pillai} showed  that $Pi(n)$ is multiplicative and in fact satisfies

$$Pi(n) = \sum_{d|n} d\phi(\frac{n}{d}),$$
and 
$$\sum_{d|n} Pi(d) = n\tau(n) = \sum_{d|n} \sigma(d) \phi(\frac{n}{d}),$$ where $\tau(n)$ is the number of divisors of $n$. For related results, see the recent work by Zhang and Zhai \cite{Z and Z}.

The second and, for our purposes, more important connection is that there is a close relationship between the size of $\sigma(n)/n$ and $\phi(n)$.   If one wants  to understand deficient, perfect and abundant numbers, then the function $h(n)= \sigma(n)/n$ is a natural object to study.
Recent work by Dubovski\cite{Dubovski} also looks at the more general function $$f_s(n) = \frac{\sum_{d|n} d^s}{n^s}.$$ Note that $h(n) = f_1(n)$.

One of the most basic properties of $h(n)$ is that $h(ab) \geq h(a) $ with equality if and only if $b=1$. An immediate consequence of this inequality is that if $n$ is perfect or abundant, then $mn$ is abundant for any $m >1$. Motivated by this last observation, a major object of interest is what are  primitive non-deficient numbers, which are non-deficient $n$ with no non-deficient divisors. Note that every perfect number is primitive non-deficient.\footnote{Some authors prefer to study primitive abundant numbers which are abundant numbers with no abundant divisor, but for many purposes primitive non-deficient numbers are a more natural set for our purposes.} 

Consider the closely related function  $$H(n) = \prod_{p|n}\frac{p}{p-1},$$ where the product is over the prime divisors of $n$. The choice of the capital letter for $H(n)$ is due to its close connection to $h(n)$. In particular, $h(n) \leq H(n)$ with equality if and only if $n=1$. Moreover, $H(n)$ is the best possible bound on $H(n)$ that only keeps track of the primes which divide $n$ but not their multiplicity. In particular, for any $n$
$$\lim_{k \rightarrow \infty} h(n^k) = H(n). $$ There is a close connection between  $H(n)$ and  $\phi(n)$. In particular, $H(n) = \frac{n}{\phi(n)}.$ Thus, we can rewrite Equation (\ref{KL(n) def}) as

\begin{equation} \kl (n) = \sum_{d|n} d\log \frac{H(d)}{2} . \label{KL def rewritten in term of H} \end{equation} When $n$ is non-deficient, $H(n) >2$ and  it is natural to look at $S(n) =H(n)-2$, which we will term the {\emph{surplus}} of $n$.  When $n$ is non-deficient, $S(n) > 2 + \frac{1}{n},$ since $H(n)>2$ and the denominator of $H(n)$ is some number which is less than $n$. Our main result uses the behavior of $\kl(n)$ to prove the following. 

\begin{theorem} Assume that $n$ is an odd primitive non-deficient number. If $n$ is a perfect number, 

$$S(n) > \frac{2\left(\log \frac{24 \sqrt{2}}{25}\right)} {\sqrt{n}} + \frac{3}{n}.$$

If $n$ is abundant then 

$$S(n) > \frac{2\left(\log \frac{24 \sqrt{2}}{25}\right)} {\sqrt{n}}.$$
\label{Main new surplus theorem with square root bound}
\end{theorem}

Theorem \ref{Main new surplus theorem with square root bound} has the following weaker but nicer corollary:

\begin{corollary} If $n$ is an odd non-deficient number then $S(n) > \frac{3}{5\sqrt{n}}$. \label{Main corollary}
\end{corollary}

Note that Corollary \ref{Main corollary} does have non-trivial content in the sense that one can provide examples of  odd $n$ which are deficient, which have $H(n)>2$, and which have $S(n) < \frac{3}{5\sqrt{n}}$. Examples include $n=165$ and $n=195$. Motivated by this, we will prove a stronger claim.

\begin{theorem} If $n$ is a odd and non-deficient number with $k$ distinct prime factors then $$H(n) \geq 2 + \frac{4}{3}{\frac{1}{n^\frac{2}{k}}}.$$ \label{tighter Surplus lower bound} 
\end{theorem}

Bounds relating the surplus to other functions of a number also exist. See in particular prior work by the first author \cite{Zelinsky reciprocal bound}, which showed a close relationship between the surplus and the sum of the reciprocals of the prime factors of an odd perfect number. The first bounds on the sum of the reciprocals of the prime factors of an odd perfect number were proven by Perisastri \cite{Perisastri} with tighter subsequent bounds appearing in a series of papers by  Cohen, Hagis and Suryanarayana  \cite{Cohen 1978,Cohen 1980,Suryanarayana III,Suryanarayana and Hagis}. 

The surplus also appeared in  another paper \cite{Zelinsky smallest component} in which the first author used the behavior of the surplus to prove that an odd primitive non-deficient number $n$ must have a surplus not much larger than its largest square-free divisor $R$. If $n=p_1^{a_1}p_2^{a_2} \cdots p_k^{a_k}$ with primes $p_1 < p_2 \cdots < p_k$ then we will write $R(n) = p_1p_2\cdots p_k$. Note that $R(n)$ is sometimes referred to as the radical of $n$.  The following proposition was proved.

\begin{theorem} (Zelinsky, \cite{Zelinsky smallest component}) Let $n$ be a primitive non-deficient number. Then there is an $i$, with $1 \leq i \leq k$ such that \label{bound with kR}
\begin{equation}\label{General inequality for primitive non-deficient numbers}p_i^{a_i+1} < 2kR.\end{equation}
\label{General p to (a+1) < r bound}
\end{theorem}

The first author also conjectured that the $2kR$ on the right-hand side of Equation (\ref{General inequality for primitive non-deficient numbers}) could be replaced with just $R$.

For the remainder of this paper, we will assume that $n=p_1^{a_1}p_2^{a_2} \cdots p_k^{a_k}$ with primes $p_1 < p_2 \cdots < p_k$ and let $R=R(n) = p_1p_2\cdots p_k$. 

Recent work on odd perfect numbers by a variety of authors have been to prove theorems of the form   ``If $n$ an odd perfect number then $R(n) \leq  Cn^\alpha$.''  A major goal of these papers is to show that $R(n) < \sqrt{n}$. For example, Pomerance and Luca \cite{Pomerance and Luca} showed that one must have $R < 2n^{\frac{17}{26}}$. Subsequently Klurman \cite{Klurman} showed that there is a constant $C$ such that if $n$ is an odd perfect number then $R < Cn^{\frac{9}{14}}$. It is not hard to show that one may take $C=4$ in Klurman's argument. Others \cite{Ellia, OchemRaoRadical}) have proven a variety of strong restrictions on how such an $n$ would behave if $R > \sqrt{n}$. 

The denominator of $S(n)$ is $(p_1-1)(p_2-1) \cdots (p_k-1)$, and thus $S(n) >  \frac{1}{R(n)}$. Iff$R(n)  < \sqrt{n}$ (whether $n$ is perfect or abundant) Theorem \ref{Main new surplus theorem with square root bound} would become completely trivial. However, there are many examples of primitive, non-deficient $n$ where $R(n) > \sqrt{n}$, and often $R$ can be  much larger. For example, $n=945$ has $R=105$. %Thus, even if the goal of showing that an odd perfect number $n$ must satisfy $R(n) < \sqrt{n}$, the main results here will still have content. 

A brief digression: Since $h(n)$ can be rewritten as $$h(n) = \sum_{d|n}\frac{1}{d},$$
there is some connection between the study of $h(n)$ and the study of Egyptian fractions. Given a rational number $x$, an Egyptian fraction representation of $x$ is list of distinct natural numbers $a_1, a_2, \cdots a_s$ such that $$x= \frac{1}{a_1} \cdots + \frac{1}{a_s}.$$ Thus, a sum giving rise to a specific value of $h(n)$ can be thought of as a very restricted type of Egyptian fraction. In that context, it is worth noting that the Kullback-Leibler divergence is closely related to notions of entropy, and can be thought of as a generalization of entropy. The use of entropy to understand Egyptian fractions has occurred in other contexts \cite{CFHMPSV} although the use there is in  ore sophisticated than that in this paper. 

There are also other recent examples in the literature of using ideas related to entropy and information to obtain number theoretic results. Kontoyiannis \cite{Kontoyiannis1, Kontoyiannis2} used information-theoretic ideas to give a proof of Chebyshev's classical result that if $f(x)=\sum_{p \leq x}  \frac{\log p}{p}$ where the sum is over primes at most $x$, then $f(x)$ is asymptotic to $\log x$.

\section{Preliminaries}

A key ingredient in proving Theorem \ref{Main new surplus theorem with square root bound} will be the following.

\begin{lemma} Let $n$ be a non-deficient number. Then $KL(n) >0$. \label{KL is positive for any non-deficient number}
\end{lemma}
\begin{proof} Our earlier remarks show this is true when $n$ is perfect. So we may assume that $n$ is abundant. We construct two distributions on $n$. The first distribution is again $P(d)= \frac{\phi(d)}{n}$. The second distribution is defined as follows. We choose a constant $\alpha$ such that

$$\sum_{d|n, H(d) \leq 2} d + \sum_{d|n, H(d) >2 } \alpha d=2n.$$ Then $Q(d)$ is a distribution where $Q(d) = \frac{d}{2n}$ when $H(d) \leq 2$, and $Q(d)= \frac{\alpha d}{2n}$ when $H(d) > 2$. 

We have that the Kullback-Liebler divergence of these two distributions is non-negative and so 

\begin{equation}\begin{split} \kl (n) & = 2n\left(\sum_{d|n} d\log \frac{d}{2\phi(d)}\right)\\ & > 2n\left(\sum_{\substack{d|n\\ H(d) \leq 2}} \frac{d}{2n}\log \frac{d}{2\phi(d)} + \sum_{\substack{d|n\\ H(d) >2}} \alpha \frac{d}{2n}\log \frac{\alpha d}{2\phi(d)}\right) \geq 0.
\end{split}
\end{equation} 
 In the above chain of inequalities, the leftmost inequality comes from the fact that $0<\alpha <1$ and that if $H(d) >0$, then $ d\log \frac{d}{2\phi(d)} > \alpha  d\log \frac{\alpha d}{2\phi(d)} .$
\end{proof}

To some extent a reader hoping that the Kullback-Liebler divergence would give new bounds about perfect numbers may be disappointed in Lemma \ref{KL is positive for any non-deficient number}, since this shows that our basic inequality for $\kl (n)$ applies to any non-deficient number. Later, we will discuss aspects which may help partially alleviate such a reader's disappointment. 

% Add back in if try to do the last bit of optimization.
%We will also include here the following Lemma which is a %straightforward calculus exercise which we will need later.

%\begin{lemma} If $A>0$ and $B>0$, then the function %$f(x)=\frac{A}{x} + Bx$ restricted to the range of $x>0$ satisfies $f(x) \geq 2\sqrt{AB}$. \label{annoying calculus lemma}
%\end{lemma}

We also will need the following technical results which we note here.

\begin{lemma} Let $f(x)=x\log \frac{x}{2(x-1)}$. Then $f(x)$ is a decreasing function for $x>1$. \label{x log function is strictly decreasing}
\end{lemma}
\begin{proof} Note that $f'(x) = \frac{1}{1-x} + \log \frac{x}{2(x-1)}$, and it is easy to see that $f'(x) <0$ for $x>0$.
\end{proof}

Finally we need the following lemma.

\begin{lemma} Let $n$ be a non-deficient number, with smallest prime factor $p$, and with $k$ distinct prime factors. Then $k \geq p$. \label{Servais's Lemma}
\end{lemma}
Versions of Lemma \ref{Servais's Lemma} are old, dating back to Servais in the late 19th century (although Servais only stated the Lemma for odd perfect numbers). A proof of this Lemma can be found in Section 5 of a prior paper by the second author \cite{Zelinsky big} where substantially tighter bounds are proved.  Recent work \cite{AslakKirf} by Aslaksen, and Kirfel further strengthens these results, and a similar strengthening is in a recent paper cite{Stone} by Stone.

\section{Proving the main results}

\begin{proof}[Proof of Theorem \ref{Main new surplus theorem with square root bound}] We note that the theorem for all odd non-deficient numbers follows if we can prove it for odd primitive non-deficient numbers. Thus, we may  assume that $n$ is an odd primitive non-deficient number.  Since $n$ is odd and non-deficient, $n$ must have at least three distinct prime divisors. Set $n= p_1^{a_1} \cdots p_k^{a_k}$ where $p_1 < p_2 \cdots p_k$. Let $j$ be the largest positive integer such that
$$H(p_1^{a_1} \cdots p_j^{a_j}) \leq H(p_{j+1}^{a_{j+1}} \cdots p_k^{a_k}).$$

We will write $a=p_1^{a_1} \cdots p_j^{a_j}$ and $b=p_{j+1}^{a_{j+1}} \cdots p_k^{a_k}$. We have \begin{equation}
    H(a)^2 \leq H(a)H(b) \leq 2+S(n),
\end{equation}

and so

\begin{equation}
    H(a) \leq \sqrt{2+S(n)} \label{upper bound for H(a)}.
\end{equation}

From the definition of $j$, we have

\begin{equation} H(a) \frac{p_{j+1}}{p_{j+1}-1} \geq H(b) \frac{p_{j+1}-1}{p_{j+1}} \label{First H(a) bound on H(b)},
\end{equation}

which is equivalent to 

\begin{equation} H(a) \geq \frac{p_{j+1}^2 -1}{p_{j+1}^2}H(b) \label{Second H(a) bound on H(b)}.
\end{equation}

We have that $p_{j+1} \geq 5$ and thus
\begin{equation} H(b) \leq \frac{25}{24}\sqrt{2+S(n)}. \label{upper bound for H(b)}
\end{equation}

Note that $a$ and $b$ are both divisors of $n$ we have

\begin{equation} a \log \frac{H(a)}{2} + b \log \frac{H(b)}{2} + \log\frac{1}{2} + E_1 +\sigma(n)\log \frac{H(n)}{2} > \kl(n) >0 . \label{application of kl to non-deficient with a and b terms}
 \end{equation}

In Equation (\ref{application of kl to non-deficient with a and b terms}) $E_1 = \log \frac{1}{2} + 3\log \frac{3}{4} + 5\log \frac{5}{8} + 7 \log \frac{7}{12}$. We are using that $1|n$, and that $n$ has at least three distinct prime divisors, along with Lemma \ref{x log function is strictly decreasing} to get that the primes being $3$, $5$ and $7$ is the weakest scenario.   The $\sigma(n)\log \frac{H(n)}{2}$ term in Equation (\ref{application of kl to non-deficient with a and b terms}) is coming from overestimating $\kl(n)$ and assuming that every divisor $d$ has $H(d) =H(n)$.

We now break this down into two cases, when $n$ is perfect and when $n$ is abundant. If $n$ is perfect, then we have $\sigma(n)=2n$, and thus we obtain from Equation (\ref{application of kl to non-deficient with a and b terms}) that 

\begin{equation}  \log \frac{H(n)}{2} > \frac{1}{2n}\left(a \log \frac{2}{H(a)} + b \log \frac{2}{H(b)} +E_2 \right). \label{Right before x > log (1+x) step}  
\end{equation}
Here $E_2= -E_1 =  \log 2 + 3\log \frac{4}{3} + 5\log \frac{8}{5} + 7 \log \frac{12}{7}$. Since $S(n) = H(n)-2$, we have that  $\log\frac{H(n)}{2} = \log \left(1 + \frac{1}{2}S\right), $ which combines with the fact that $x \geq \log(1+x)$ to obtain from Equation (\ref{Right before x > log (1+x) step}) that

\begin{equation}
    S(n) \geq \frac{1}{n}\left(a \log \frac{2}{H(a)} + b \log \frac{2}{H(b)} +E_2 \right).  \label{lower bound for S first version}
\end{equation}
We now apply Equation (\ref{upper bound for H(a)}) and Equation (\ref{upper bound for H(b)}) to Equation (\ref{lower bound for S first version}) to get that
\begin{align} 
S(n) & \geq \frac{1}{n}\left(a \log \frac{2}{\sqrt{2+S(n)}} + b\log \frac{48}{25 \sqrt{2+S(n) }} +E_2\right)   \\ 
 &> \frac{a+b}{n} \log \frac{48}{25 \sqrt{2+S(n) }} + \frac{E_2}{n}.\label{lower bound for S second version }
\end{align}
Since $ab=n$, $a+b \geq 2\sqrt{n}$, and thus
\begin{equation}
    S(n) >  \frac{2}{\sqrt{n}} \left(\log \frac{48}{25\sqrt{2}}  - \frac{1}{2}\log ( 1+\frac{1}{2}S(n)) \right) +\frac{E_2}{n},
\end{equation}
which implies that
$$S(n) > \frac{2\left(\log \frac{24 \sqrt{2}}{25}\right)}{\sqrt{n}} + \frac{3}{n} $$
which is the desired inequality.

The case for when $n$ is abundant is nearly identical. The only change is that we do not have $\sigma(n) =2n$. We instead use that $\sigma(n) < n(2+S(n))$. We proceed essentially the same as before to obtain that

\begin{equation}
S(n) > \frac{4\left(\log \frac{24 \sqrt{2}}{25}\right)}{(2+S(n))\sqrt{n}} + \frac{E_2}{(2+S(n))n},
\end{equation}

which implies the desired result.  

% If try to squeeze that last bit out go back to this. 
%We now note that $ab=n$ and so apply Lemma \ref{annoying calculus lemma} to the right-hand side of Equation \ref{lower bound for S second version } with $b=x$, $A=n\log \frac{2}{\sqrt{2+S(n)}} $ and $B= \log \frac{48}{25 \sqrt{2+S(n) }},$
%to obtain that
%\begin{equation}
 %   S(n) \geq \frac{1}{n} \sqrt{n} \sqrt{\left(\log \frac{2}{\sqrt{2+S(n)}}\right) \left(\log \frac{48}{25 \sqrt{2+S(n) }}\right)}
%\end{equation}
\end{proof}

We now turn to proving Theorem  \ref{tighter Surplus lower bound}. We first need an additional lemma. 

\begin{lemma} Assume that $n$ is a non-deficient number, and assume that $n=p_1^{a_1}p_2^{a_2} \cdots p_k^{a_k}$ where $p_1 \cdots p_k$ are distinct primes. Let $j$ be the number where $1 \leq j \leq k$, and where for any $i$ with $1 \leq i \leq k$, $p_j^{a_j+1} \leq p_i^{a_i+1}$. Then we have two inequalities depending on whether $n$ is even or odd. If $n$ is odd, then $$H(n) \geq 2 + \frac{4}{3}\frac{1}{p_j^{a_j+1}}. $$

If $n$ is even, then either $p_j=2$ or $$H(n) \geq 2 + \frac{7}{4}\frac{1}{p_j^{a_j+1}}. $$     \label{p sub lemma}
\end{lemma}

\begin{proof} We first note that for any prime $p$ and any positive integer $a$, we have that

$$H(p^a) = 1 + \frac{1}{p} + \frac{1}{p^2} + \frac{1}{p^3} \cdots = h(p^a) + \frac{1}{p^{a+1}} + \frac{1}{p^{a+2}} \cdots, $$  

and thus $H(p^a) > h(p^a) + \frac{1}{p^{a+1}}$. Thus, we have

\begin{equation}
    H(n) > \left(\prod_{\substack{1 \leq i \leq k \\ i \neq j}} h(p_i^{a_i})\right) \left(h(p_j^{a_j}) + \frac{1}{p_j^{a_j+1}} \right) = h(n) + \frac{1}{p_j^{a_j+1}}\prod_{\substack{1 \leq i \leq k \\ i \neq j}}{h(p_i^{a_i})}.\label{central inequality}
\end{equation}

We now split into two cases, when $n$ is even and when $n$ is odd. First consider the case when $n$ is odd. We note that for any odd prime $p$, $H(p^a) = \frac{p}{p-1} \leq \frac{3}{2}$. Thus, we have 
that $$\frac{3}{2} \prod_{\substack{1 \leq i \leq k \\ i \neq j}}{h(p_i^{a_i})} \geq 2,$$ (with possible equality only if $p_j=3$). This implies that 

\begin{equation} \label{lower bound for big product in case n is odd} \prod_{\substack{1 \leq i \leq k \\ i \neq j}}{h(p_i^{a_i})} \geq \frac{4}{3}.
\end{equation} Combining Inequality \ref{central inequality} with Inequality \ref{lower bound for big product in case n is odd} gives our desired inequality.

We now consider the case where $n$ is even, and assume that $p_j \neq 2$. Then $$\prod_{\substack{1 \leq i \leq k \\ i \neq j}}{h(p_i^{a_i})} \geq h(4) = \frac{7}{4}.$$
We can use $h(4)$ rather than $h(2)$ in the above since $2^2|n$ and otherwise we would have $p_j=2$. 
\end{proof}

\begin{proof}[Proof of Theorem \ref{tighter Surplus lower bound}] Assume $n$ is odd and non-deficient with $n=p_1^{a_1} \cdots p_k^{a_k}$. Let $p_j$ be defined as before, and set $R= \mathrm{rad}(n)$. Note that $$(p_j^{a_j+1})^k < p_1^{a_1+1}p_2^{a_2 +1} \cdots p_k^{a_k+1} = nR.$$ Thus, $$p_j^{a_j+1} \leq (nR)^\frac{1}{k} \leq (n^2)^\frac{1}{k},$$
from which the result follows when we combine with the inequality from Lemma \ref{p sub lemma}.
\end{proof}

\section{Related results and open questions}

Given our main theorems, the following question seems very natural. 

\begin{question} What is the slowest growing class of functions $f(x)$ such that 
\begin{enumerate}
    \item $f(x)$ is a positive increasing function,
    \item $S(n) > f(x)^{-1}$?
\end{enumerate}
\end{question}

One direction to go in is to look at the function
$$g(n) = \frac{S(n)}{\frac{1}{n^{1/2}}} = S(n)n^{1/2}.$$ 

Let $B_1$ be the set of odd non-deficient record setters for $g(n)$. That is, $n \in B_1$ if and only if for all $k< n$, $g(k) < g(n)$. The first few numbers in $B_1$ are
$$945, 1575, 2205, 2835, 3465, 5775, 8085, 10395, 12705, 15015, 19635, 21945, 25935, 26565  $$ $$31395, 33495, 35805, 42315, 42735, 45045, 58905, 65835, 75075, 98175,  105105, 135135 \cdots $$ % $$   165165, 195195, 225225, 255255, 285285,  345345,  373065,  435435,  465465, 555555, 615615 \cdots$$ 

Table \ref{table:Elements of B_1} in Appendix A shows the first few elements of $B_1$ along with their corresponding values for $g(n)$ and $h(n)$.  This sequence is not currently in the OEIS. Note that $4095$ is the first odd non-deficient number not in $B_1$. 

A related set is $B_2$ defined by $n \in B_2$ if and only if
\begin{enumerate}
    \item $n$ is odd and primitive non-deficient.
    \item For all $k<n$ where $k$ is odd, primitive non-deficient, we have $g(k) < g(n)$.
\end{enumerate}

The first few elements of $B_2$ are:
$$945,1575,2205, 3465, 5775, 8085, 12705, 15015, 19635, 21945, 25935,  26565, 31395,  \cdots $$
$$33495, 35805,  42315, 42735, 47355, 49665, 54285, 61215, 68145, 70455, 77385, \cdots $$ % 77385, 82005, 84315 ,91245 \cdots$$ 
%$$95865, 102795, 112035, 116655, 118965, 123585, 125895, 130515, 146685, 151305, 158235.$$
Note that $47355$ is the first element of $B_2$ which is not an element of of $B_1$, and it seems likely that all subsequent elements of $B_2$ are not elements of $B_1$. Table \ref{table:Elements of B_2} in Appendix A lists the first few elements of $B_2$, and the corresponding $g(n)$ and $h(n)$ values.

Theorem \ref{tighter Surplus lower bound} implies that $B_2$ is an infinite set. It appears that all elements of $B_2$ are divisible by $15$, but we suspect that this will break down for very large values. Let $b_{1,j}$ denote the $j$th element of $B_1$ and let $b_{2,j}$ denote the the $j$th element of $B_2$. There are six related questions where we do not know the answers.

\begin{enumerate}
    \item For any $\epsilon>0$, is there a $t$ such that $g(b_{1,j+1}) < g(b_{1,j})$?
     \item For any $\epsilon>0$, is there a $j$ such that $g(b_{2,j+1}) < g(b_{2,j})$?
     \item Are there infinitely many $j$ such that $h(b_{1,j}) < h(b_{1,j+1})$?
     \item Are there infinitely many $j$ such that $h(b_{2,j}) < h(b_{2,j+1})$?
     \item Can $g(b_{1,j+1}) - g(b_{1,j})$ grow arbitrarily large?
     \item Can $g(b_{2,j+1}) - g(b_{2,j})$ grow arbitrarily large?
\end{enumerate}

We now pivot slightly and examine what can we conclude about those $n$ where $\kl(n)>0$. 

\begin{lemma} If $\kl(n)>0$, then $H(n) > 2$. \label{kl positive forces H greater than 2}
\end{lemma}
\begin{proof} If $H(n) < 2$, then $H(d)< 2$ for any $d$ where $d|n$. Hence, $\log \frac{d}{2\phi(d)} < 0$ for all $d$, and thus $\kl(n) <0$. Thus, $H(n) \geq 2$. Now, note that $H(n)=2$ exactly when $n$ is a power of 2, and it is easy to check in this case that $\kl(n) <0$.
\end{proof}

\begin{proposition} If $\kl (n) >0$ then $h(n) \geq \frac{12}{\pi^2}$.\label{easy lower bound on h assuming n in A}
\end{proposition}
\begin{proof} Assume that $n \in A $. By Lemma \ref{kl positive forces H greater than 2}, we have that $H(n) \geq 2$. 

We have that $h(n) \geq \prod_{p|n}\frac{p+1}{p}$. Thus, from the definitions of $h(n)$ and $H(n)$,

\begin{equation}\frac{h(n)}{H(n)} \geq  \prod_{p|n}\frac{\left(\frac{p+1}{p}\right)}{\left(\frac{p}{p-1}\right)} =  \prod_{p|n}\frac{p^2-1}{p^2}  > \prod_{p} \frac{p^2-1}{p^2}. \label{chain to get zeta}\end{equation} Here the first two products in our chain of inequalities are over primes which divide $n$ and our last is over all primes $p$. We have that $$\prod_{p} \frac{p^2-1}{p^2} = \frac{1}{\zeta(2)} = \frac{6}{\pi^2 },$$ and thus we may use Equation (\ref{chain to get zeta}), to obtain that 

\begin{equation} h(n) > H(n)\frac{6}{\pi^2 }.   
\end{equation}
Since $H(n) > 2$, this implies the desired result. 
    
\end{proof}

We can tighten the above slightly with the next two propositions. 

\begin{proposition}  If $n$ is odd and $\kl (n)>0$ then $h(n) \geq \frac{16}{\pi^2}$. \label{lower bound for h when n odd}    
\end{proposition}
\begin{proof} The argument is the same as in the proof of Proposition \ref{easy lower bound on h assuming n in A}, except that we assume that $2$ is not in our product. Thus, we may remove a factor of $\frac{3}{4}$ from the product, and thus we replace  $\frac{6}{\pi^2 }$ in our argument with $\frac{6}{\pi^2 }\frac{4}{3}$. 
\end{proof}

\begin{proposition} If $\kl (n)>0$, then $h(n) \geq \frac{27}{2\pi^2}$.    
\end{proposition}
\begin{proof} By Proposition \ref{lower bound for h when n odd}, we may assume that $n$ is even. If $3|n$, then $h(n) \geq 2$, since then $6|n$. Thus, we may assume that $3$ does not divide $n$. We then multiply our product by a factor of $\frac{9}{8}$ which gives us the desired result. 
\end{proof}

We will define the set $T$ as follows: $n$ is in $T$ if 
\begin{enumerate}
    \item $\kl(n) >0$.
    \item For any $m$ where $1 \leq m < n$, and $\kl(m) >0$, we have $h(m) > h(n)$.
\end{enumerate}

Essentially $T$ is the set elements of $A$ which set a new record for how small $h(n)$ can be and still have $\kl(n)>0$. Note that $T$ is finite if and only if there is a number $n_0$ such that $$h(n_0)= \inf_{\kl(n)>0} h(n) .$$

The first few elements of $T$, and their corresponding values of $\kl(n)$ and $h(n)$, are shown in Table \ref{table:Elements of T}. The elements of $T$ do not seem to be in the OEIS. In Appendix B, we plot the graph of points of the form $(h(n),\kl(n))$ for $n \leq 5000$.

\begin{table}[hbt!]
\centering
\begin{tabular}{|c|c|c|}
\hline
$n$&  $\kl(n)$ &  $h(n)$ \\
\hline
6 &  0.876597250733 & 2 \\
110 & 9.1306578375 & 1.96363636364 \\
130 & 7.36577620456 & 1.93846153846 \\
170  & 3.93999763019 &  1.90588235294 \\
190 & 2.2563934015 & 1.89473684211 \\
2950 & 9.51857709771 & 1.89152542373 \\
3050 & 6.68286526629 & 1.89049180328 \\
7826 & 17.0588059708 & 1.88908765653 \\
15554 & 30.9396278647 & 1.88864600746 \\
15862 & 25.8430364891 & 1.88828647081 \\
16478 & 15.6530193347 & 1.88760771938 \\
16786 & 10.5594766699  & 1.8872870249 \\
17402 & 0.375062167248 & 1.88667969199 \\
19270 & 2.4222880012 & 1.88313440581 \\
89470 & 11.9467523362 & 1.88308930368 \\ 
91310 & 3.3076383573  & 1.88299200526 \\
36984230&  25.7491698638 & 1.88295649254 \\
\hline
\end{tabular}
\caption{The first few elements of T}
\label{table:Elements of T}
\end{table}

Similarly, we may define $T_o$ as follows: $n$ is $U$ if

\begin{enumerate}
    \item $\kl(n) >0$.
    \item For any odd $m$ where $1 \leq m < n$, $m$ odd, and $\kl(m) >0$, we have $h(m) > h(n)$.
\end{enumerate}

The first few elements of $T_o$ are:
$$315, 1365, 49335, 89355, 148155, 158865,  3853815, 12109965.$$

Note that one might conjecture that the elements of $T_o$ are square-free, but 3853815 is divisible by $17^2$. The first few elements of $T_o$, along with their corresponding values of $\kl (n)$ and $h(n)$ are given in Table \ref{table:Elements of TsubO}. We strongly suspect that both $T$ and $T_0$ are infinite.

\begin{table}[hbt!]
\centering
\begin{tabular}{|c|c|c|}
\hline
$n$&  $\kl(n)$ &  $h(n)$ \\
\hline
315 &  1.22497946997 & 1.98095238095 \\
1365 & 16.4285888516 & 1.96923076923 \\
49335 & 115.855679456 & 1.9614472484 \\
89355  & 186.312257706 &  1.95964411617 \\
148155 & 296.139310959 & 1.95946137491 \\
158865 & 67.5102085587 & 1.95788877349 \\
3853815 & 430.225703476 & 1.9577566645 \\
12109965 & 1654.86520391 & 1.95773976225 \\
16100805 & 206.792005895 & 1.95758659272 \\
29220765 & 468.355284063 & 1.957533966  \\
58981755 & 3331.94336991 & 1.95752140641 \\
59320905 & 3009.99338375  & 1.95751565152 \\
61723515 & 147.109626597 & 1.95746726349 \\

\hline
\end{tabular}
\caption{The first few elements of $T_o$}
\label{table:Elements of TsubO}
\end{table}                              

We can also make an analogous definition to primitive non-deficient numbers for $KL(n)$. We first need one motivating proposition.

\begin{proposition} If $\kl(n) \geq 0$, then for any $m>1$, $\kl(mn)>0$. \label{kl(n) >0  implies kl(mn)>0}
\end{proposition}
\begin{proof} 

The claim is equivalent to showing that if $\kl(n) \geq 0$, then for any prime $p$, $\kl(np)>0$. We consider two cases, $p \nmid n$ and $p|n$.
First, let us consider the case where $p \nmid n$.

We have then $$\kl(pn) = \sum_{d|n}  d\log \frac{d}{2\phi(d)} + pd \log \frac{pd}{(p-1)\phi(d)} = \kl(n) + p\sum_{d|n} d\log \frac{pd}{(p-1)\phi(d)}.$$

Since $p\sum_{d|n} d\log \frac{pd}{(p-1)\phi(d)} > p\sum_{d|n} d\log \frac{d}{\phi(d)} \geq 0$, we are done with this case.

Now, assume that $p|n$. Set $n=n_0 p^m$ where $p\nmid n_0$. Then 
\begin{equation}\kl(np) = \kl(n) + p^{m+1}\sum_{d|n_0}d \log \frac{dp}{2\phi(d)(p-1)}. \label{kl(mn) lemma second case equation 1}\end{equation}

Now, if $\kl(n_0) \geq 0$, then the second term on the right-hand side of Equation \ref{kl(mn) lemma second case equation 1} is positive, since 
$$\sum_{d|n_0}d \log \frac{dp}{2\phi(d)(p-1)} > \kl(n),$$ and so we are done. Thus, we may assume that $\kl(n_0)<0$. However, 
in that case $$\kl(n) = \kl(n_0) + \left(\sum_{i=1}^m\right)\sum_{d|n_0} d\log \log \frac{dp}{2\phi(d)(p-1)} >0, $$
    and so $\sum_{d|n_0}d \log \frac{dp}{2\phi(d)(p-1)} >0$, and thus we are again done.
\end{proof}

Proposition \ref{kl(n) >0  implies kl(mn)>0} motivated the following definition:

\begin{definition} We say a number $n$ is $KL$-primitive number if $\kl(n) \geq 0$, and $\kl(d)<0$ for any $d|n$ with $d<n$. 
\end{definition}

The $KL$-prmitive numbers under 1500 are 

$$6,20, 28,70,88, 104, 110, 130, 136, 152, 170, 190, 315, 368, 464, 496, 572, 592,656, 688,  $$
$$ 748, 836, 884, 988, 1012, 1078, 1150, 1155, 1196, 1276, 1292, 1364,
 1365, 1450. $$

There are $KL$-primitive numbers which are themselves deficient. In particular, it is not too hard to show the following that if $k \geq 4$ and $p$ is a prime where $p=2^k + i$ for either $i=1$ or $i=3$, then $n=2^{k-1}p$ satisfies $\kl(n)>0$. Any such $n$ is itself deficient, and thus such an $n$ must have a $KL$-primitive divisor (possibly $n$ itself). However, it is unclear if there are infinitely many such numbers. The standard conjectures are that there are only finitely many primes of the form $2^k+1$, the so-called Fermat primes, but that there are infinitely many primes of the form $2^k+3$. Thus it seems safe to conjecture the following.

\begin{conjecture} There are infinitely many $n$ which are deficient and $KL$-primitive.
\end{conjecture}

There are two additional questions which we are less confident of the answers.

\begin{enumerate}
    \item Are there infinitely many $n$ which are both primitive non-deficient and $KL$-primitive?
    \item Are there infinitely many $n$ which are odd, deficient and $KL$-primitive? 
\end{enumerate}

One direction for improvement of our main results is to use results in the literature which give lower bounds for the Kullback-Liebler divergence. For example,  Pinsker's inequality is a tightening of the claim that the Kullback-Liebler divergence is non-negative.

\begin{lemma}(Pinsker's Inequality) Let $P$ and $Q$ be discrete probability distributions  on the same sample space $\mathcal{X}$. Let $M=\max_{x \in X}{|P(x)- P(q)|}$. \label{Pinsker}
\begin{equation}
    D_\kl(P \parallel Q) \geq 2M^2. 
\end{equation}    
\end{lemma}

Note that although it is not in general possible to reverse Pinsker's inequality, in the sense of bounding the Kullback-Liebler divergence by a function of the $M$, work by  Berend, Harremo{\"{e}}s, and Kontorovich \cite{BHK} shows a broad set of contexts where such a reverse inequality does apply. It is unclear if the distributions we are using satisfy their criteria.  Using Lemma \ref{Pinsker}, we obtain the following.
\begin{proposition}
If $n$ is a perfect number then
  \begin{equation}
      \kl(n) \geq \frac{(2\phi(n)-n))^2}{2n}. \label{Pinsker version of KL}
  \end{equation}    \label{Pinsker prop version of KL}
\end{proposition}

\begin{proof} The proof is identical to the proof of  Lemma \ref{KL is positive for any non-deficient number} except we also use that 
$$\max_{d|n} |\frac{\phi(d)}{n}- \frac{d}{2n}| \geq \frac{2\phi(n)-n}{n}.$$ \end{proof}

One question which remains open is whether Inequality \ref{Pinsker version of KL} is satisfied by all odd non-deficient numbers.

\section{Other weights and other inequalities}

Since a major aspect of our approach was to use an inequality arising from probability to understand perfect numbers and more generally non-deficient numbers, two obvious questions arise. First, can we gain other interesting results by using other weight choices on the divisors and then apply that the Kullback-Liebler divergence must be positive?  Second, can we use other functions similar to the Kullback-Liebler divergence arising from probability and information theory to prove other non-trivial inequalities related to the divisors of a number?

Regarding the first question, one can make a similar inequality using the Kullback-Liebler divergence and with any other function one likes. One just has to normalize the functions so that the relevant sum over all the divisors is 1. However, this approach worked particularly well with the choice $\frac{\phi(d)}{n}$ and perfect numbers for two reasons. First, because the $n$s canceled, and second, because one then ended up with $\log \frac {H(n)}{2}$. So it seems reasonable to look for other arithmetic functions where there is similar cancellation or where the function $f(d)$ and $\sum_{d|n}f(d)$s are particularly nice.

Given this sort of constraint, what other well-behaved probability distributions on the divisors do we have? One obvious probability distribution is the uniform distribution on the divisors. We  obtain the following result. 

\begin{proposition} Let $n$ be a positive integer, and let $\tau(n)$ be the number of positive divisors of $n$. Then 
$$\sum_{d|n}\phi(d)\log \phi(d) \geq n(\log \frac{n}{\tau(n)}).  $$ \label{consequence of using p i with phi and q i as tau}
    
\end{proposition}
\begin{proof} Fix a positive integer $n$. We get two probability distributions on the divisors of $n$. The first, our usual $\phi(d)/n$ we assign as $p_i$, and the other, the uniform distribution we assign to $q_i$, so for all $i$, with $1 \leq i \leq \tau(n)$, we have $q_i = \frac{1}{\tau(n)}$. Then positivity of the Kullback-Liebler divergence yields that

\begin{equation}\label{Immediate version of KL with tau and phi} \sum_{d|n} \phi(d) \log  \frac{\phi(d)\tau(d)}{n}  \geq 0. \end{equation}
Inequality \ref{Immediate version of KL with tau and phi} then implies that

\begin{equation} \sum_{d|n} \phi(d) \log \phi(d) \geq \sum_{d|n} \phi(d) \log\frac{n}{\tau(n)} = n \log \frac{n}{\tau(n)}.
\end{equation}
\end{proof}

The authors are not aware of another way of proving Proposition \ref{consequence of using p i with phi and q i as tau}. Proposition \ref{consequence of using p i with phi and q i as tau} raises the question of whether it is true in general that \begin{equation}
    \phi(n)\log \phi(n) \geq n \log \frac{n}{\tau(n)} \label{false phi ineq}.
\end{equation}

If Inequality \ref{false phi ineq} were true for all $n$, then Proposition \ref{consequence of using p i with phi and q i as tau} would have very little content since it would have many extra terms in the sum on the left-hand-side. However, Inequality \ref{false phi ineq} is in fact sometimes false, for example it fails for $n=16$. 

It is not hard to also use Pinsker's inequality to prove the following stronger statement.

\begin{proposition} If $n >1$ then
$$\sum_{d|n}\phi(d)\log \phi(d) \geq n(\log \frac{n}{\tau(n)}) + \frac{n}{2 \tau(n)^2}.  $$    
\end{proposition}
\begin{proof} The inequality can be verified easily for $n \leq 12$. If $n \geq 12$, then it is not hard to show that $$\frac{1}{\tau(n)} - \frac{1}{n} \geq \frac{1}{\tau(n)}. $$ We then we apply Pinsker's inequality and then proceed as in the proof of Proposition \ref{consequence of using p i with phi and q i as tau}.
\end{proof}

Unfortunately, other choices of weights seem  to lead to trivial results or to the same results as above. The other obvious weight choice for generic $n$ is $\phi(\frac{n}{d})/n$. For perfect numbers we also have the corresponding weight choice $\frac{1}{2d}$.
The other combinations of these weights do not generally result in anything interesting. For example, the choice of $p_i = \frac{1}{2d}$ and $q_i = \frac{\phi(d)}{n}$ results in the inequality 

\begin{equation} \sum_{d|n} \frac{1}{d} \log \frac{n}{2 d \phi(d)} \geq 0 \label{Inequality that is trivial for Zaremba reasons}.\end{equation}

But it is not hard to show that Inequality \ref{Inequality that is trivial for Zaremba reasons} is in fact true for all $n >3$ (without any assumption about $n$ being perfect). 

We note here one other set of weights where the apparent weakness of the result is itself conjectural. For a positive integer $n$, define $v(n)$ by
$$v(n)= \sum_{d|n, d>1} \frac{1}{d} \log \frac{\tau(n)-1}{d}. $$

\begin{proposition} Let $n$ be a perfect number. Then $v(n) > 0$.\label{first v(n) bound}
\end{proposition}
\begin{proof} Note that if $n$ is a perfect number then $\sum_{d|n, d>1} \frac{1}{d}=1$. Thus, if $n$ is perfect we may use as our sample space all the positive divisors greater than 1. We assign probability $\frac{1}{d}$ to $d$ for our $p_i$, while for our $q_i$, we assign the uniform distribution and so always use $q_i=\frac{1}{\tau(n)-1}$ for all $i$. (Here we have only $\tau(n)-1$ terms.) The result then follows immediately from applying the Kullback-Liebler divergence to these distributions. 
\end{proof}

The set of $n$ where $v(n) < 0$ is rare, but not non-existent. We can use the Pinsker version to tighten Proposition \ref{first v(n) bound}, but it requires some additional work. 

\begin{proposition} Let $n$ be a perfect number. Let $p$ be the smallest prime factor of $n$. Then \label{tau -1 and 1/d set with Pinsker}
\begin{equation} v(n) \geq \frac{1}{2p^2}. 
\end{equation}
\end{proposition}
\begin{proof} For even $n$, this is a straightforward computation using the Euclid-Euler theorem for even perfect numbers. We thus may assume that $n$ is odd. We want to get a lower bound for $M$ in the Pinsker inequality. Let $k$ be the number of distinct prime factors of $n$, and let $p$ be the smallest prime factor of $n$. Then by Lemma \ref{Servais's Lemma}, $k \geq p$, and thus $\tau(n) -1 \geq 2^p -1$. Therefore,  we have $M \geq \frac{1}{2p}$, and the result follows.    
\end{proof}

However, the content of Proposition \ref{tau -1 and 1/d set with Pinsker} is empirically still weak. In fact, we conjecture the following:

\begin{conjecture} Assume that $n$ is a positive integer and assume that  $$n \not \in \{1,12,24,30,36,48,60,72,120,180,240,360\}.$$ Let $p$ be the smallest prime factor of $n$. Then $v(n) \geq \frac{1}{p^2}$.
\end{conjecture}

Regarding the second question, the authors do not know whether we can use other
 functions similar to the Kullback-Liebler divergence arising from probability and information theory to prove other non-trivial inequalities related to the divisors of a number. The set of such inequalities in the literature is vast, and we mention only a few of them here.

There is one example here that seems like a somewhat obvious choice that does not seem to work. However, we are at least in the nice position of being able to explain in part why it does not work. 

%There are also two variants of Pinsker's inequality which are worth noting. 

Tops{\o}e \cite{Topsoe} investigated a variation of Kullback-Liebler divergence which had the virtue of being a metric. This variant was also studied by Endres and Schind \cite{Endres and Schind}. Given two distributions, $P$ and $Q$ on a discrete sample space $X$, define $\dest$ by 
\begin{equation}
    \dest = \sqrt{\sum_{x \in X} P(x) \log \frac{2P(x)}{P(x) + Q(x)} +  Q(x) \log \frac{2Q(x)}{P(x) + Q(x)}}.
\end{equation}

At first sight, $\dest$ looks very promising since it is a metric and  has a symmetry that $D_{\kl}$ lacks. However, it turns out to be not useful for our purposes. In particular, if one plugs into $D_{\kl}$ values which are not from a probability distribution, then one can have a value of $D_{\kl}$ which is negative. In contrast, $\dest$ is always non-negative even for such choices since for all positive $p$ and $q$ we have

$$p \log \frac{2p}{p+q} + q \log \frac{2q}{p+q} \geq  0,$$
with equality if and only if $p=q$. Thus, the non-negativity of $\dest$ cannot tell us anything useful by itself.  However, it is plausible that there are non-trivial lower bounds for $\dest$ that can be used to get a non-trivial result here. 

There are also some similar inequalities proven by Dragomir and Glu{\u{s}}{\u{c}}evi{\'{c}} \cite{Dragomir and G} which do look promising. That work contains both other lower bounds for the Kullback-Liebler divergence as well as non-trivial upper bounds for the Kullback-Liebler divergence. The earlier mentioned bounds of Berend, Harremo{\"{e}}s and Kontorovich  may also be of interest.

Tops{\o}e \cite{Topsoe} also introduced tighter version of Pinsker's inequality. The inequality he introduced is always at least as tight as Pinsker's inequality, and often is much tighter. However, estimating the relevant quantities in Tops{\o}e's inequality seem to be difficult in the context of sums involving divisors similar to ours. There is also another set of refinements of Pinsker's inequality by Fedotov,  Harremo{\"{e}}s and Tops{\o}e \cite {FHT}.
\newpage 
\appendix
\section{Appendix: Tables of sequences}
\FloatBarrier 
\begin{table}[hbt!]
\centering
\begin{tabular}{|c|c|c|}
\hline
$n$&  $g(n)$ &  $h(n)$ \\
\hline
945 &  5.76390980585227 & 2.03174603174603\\
1575 & 7.44117556236916 & 2.04698412698413 \\
2205 & 8.80451766140542 & 2.01632653061225 \\
2835  & 9.98338463398060 & 2.04867724867725 \\
3465 & 23.9136018107373 & 2.16103896103896 \\
5775 & 30.8723271869404 & 2.06129870129870 \\
8085 & 36.5286301455790 & 2.03042671614100 \\
10395 & 41.4195733281682 & 2.21645021645022 \\
12705, & 45.7910612353820 & 2.00991735537190 \\
15015 & 74.3510943122102 & 2.14825174825175 \\
19635 & 77.9992365432936 & 2.11214667685256 \\
21945 & 79.9844749362677 & 2.09979494190021 \\
25935 & 80.7547273307130 & 2.07287449392713 \\
26565 & 84.0405480747426 & 2.08153585544890 \\
31395 & 84.6081564871294 & 2.05484949832776 \\ 
33495 & 90.0783807861720 & 2.06359158083296\\
35805 &  92.0486633745241 & 2.05915374947633\\
42315 & 92.3177554091127 & 2.03275434243176 \\ 
42735 & 97.7993295018012 & 2.04871884871885 \\
45045&  128.779872947093 & 2.32727272727273\\
58905 & 135.098640644568 & 2.28815889992361 \\ 
65835 & 138.537174406335 & 2.27477785372522 \\
75075 &  166.254101083600 & 2.21986013986014 \\
98175 & 174.411595103890 & 2.18255156608098 \\ 
105105 & 196.714505255617 & 2.18661338661339 \\
135135&  223.053282936631 & 2.38694638694639 \\
165165 & 246.594682585929 & 2.16452638270820 \\ 
195195 & 268.076682929533 & 2.16005532928610 \\
225225&  287.960550043487 & 2.40484848484848\\
255255 & 388.870712612747 & 2.27461949814891 \\ 
285285 & 401.440598235671 & 2.26131762973868 \\
345345 & 426.206983467123 & 2.24165399817574 \\
373065 & 426.744774912344 & 2.22331229142375 \\ 
435435 & 461.826570490029 & 2.22232939474319 \\
465465&  473.251682636947 & 2.21755019174374 \\
555555 & 506.231475297111 & 2.20631260631261 \\ 
 615615 & 527.211844860323 & 2.20064813235545 \\
\hline
\end{tabular}
\caption{The first few elements of $B_1$}
\label{table:Elements of B_1}
\end{table} 

\FloatBarrier 

\begin{table}[hbt!]
\centering
\begin{tabular}{|c|c|c|}
\hline
$n$&  $g(n)$ &  $h(n)$ \\
\hline
945 &  5.76390980585227 & 2.03174603174603\\
1575 & 7.44117556236916 & 2.04698412698413 \\
2205 & 8.80451766140542 & 2.01632653061225 \\
3465 & 23.9136018107373 & 2.16103896103896 \\
5775 & 30.8723271869404 & 2.06129870129870 \\
8085 & 36.5286301455790 & 2.03042671614100 \\
12705, & 45.7910612353820 & 2.00991735537190 \\
15015 &74.3510943122102 & 2.14825174825175 \\
19635 &  77.9992365432936 & 2.11214667685256\\
21945 & 79.9844749362677 & 2.09979494190021\\
25935 &  80.7547273307130 & 2.07287449392713\\
26565 & 84.0405480747426 & 2.08153585544890\\
31395 & 84.6081564871294 & 2.05484949832776\\
33495 & 90.0783807861720 & 2.06359158083296\\
35805 & 92.0486633745241 & 2.05915374947633\\
42315 & 92.3177554091127 & 2.03275434243176\\
42735 & 97.7993295018012 & 2.04871884871885\\
47355 & 101.495615581659 & 2.04345898004435\\
49665 & 103.303253061552 & 2.04119601328904\\
54285 & 106.840493139847 & 2.03724785852445\\
61215 & 111.961982832012 & 2.03244302866944 \\
68145 & 116.879958799534 & 2.02861545234427 \\
70455 & 118.477396689007 & 2.02750691931020 \\ 
77385 & 123.153310448193 & 2.02457840666796 \\
82005 & 126.179644884597 & 2.02290104261935 \\ 
84315 &  127.667217418345 & 2.02213129336417\\
91245 & 132.033673728308 & 2.02005589347361 \\ 
95865 & 134.869093350034 & 2.01883899233297 \\ 
102795 &  139.017351069266 & 2.01721873631986 \\ 
 112035 & 144.368194843552 & 2.01537019681350 \\ 
 116655 & 146.972350880497 & 2.01455574128841 \\
118965 & 148.257654602084 & 2.01417223553146 \\
123585 & 150.796072590133 & 2.01344823400898 \\
125895 & 152.049699369013 & 2.01310615989515\\
130515 & 154.526998031331 & 2.01245833812205\\
146685 &  162.905764961649 & 2.01051232232335\\
151305 & 165.222754737438 & 2.01003271537623\\
158235 & 168.639287831050 & 2.00936581666509\\
\hline
\end{tabular}
\caption{The first few elements of $B_2$}
\label{table:Elements of B_2}
\end{table}

\FloatBarrier 
\vspace{1 in}
\section{Graphs of $h(n)$ and $\kl(n)$}

\begin{figure}[htp]
    \centering
    \includegraphics[width=13.4 cm]{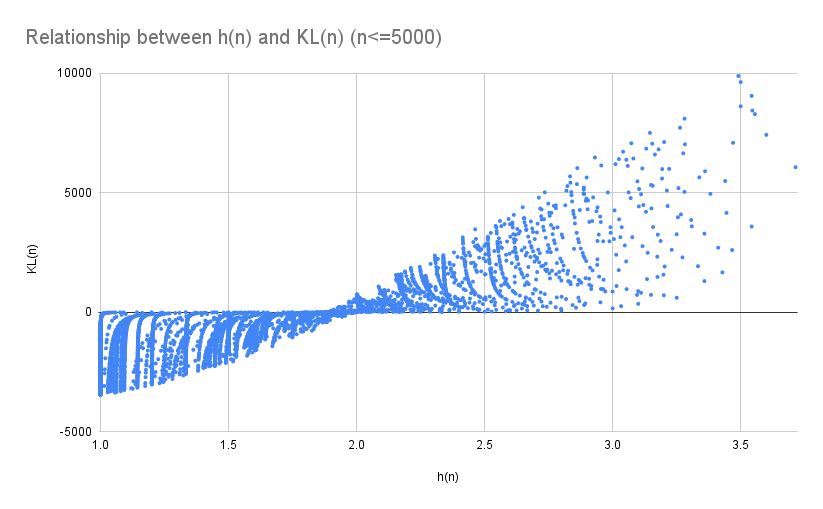}
    \caption{{$h(n)$ and $\kl(n)$}}
    \label{KL graph 1}
\end{figure}

\FloatBarrier 

\begin{figure}[htp]
    \centering
    \includegraphics[width=13.4 cm]{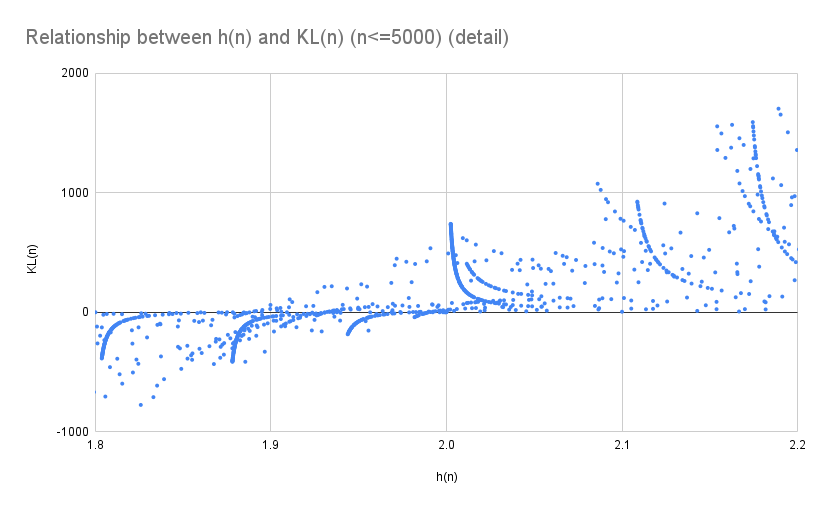}
    \caption{{$h(n)$ and $\kl(n)$ zoomed in near $h(n)=2$}}
    \label{KL graph 2}
\end{figure}

\FloatBarrier 

\begin{figure}[htp]
    \centering
    \includegraphics[width=13.4 cm]{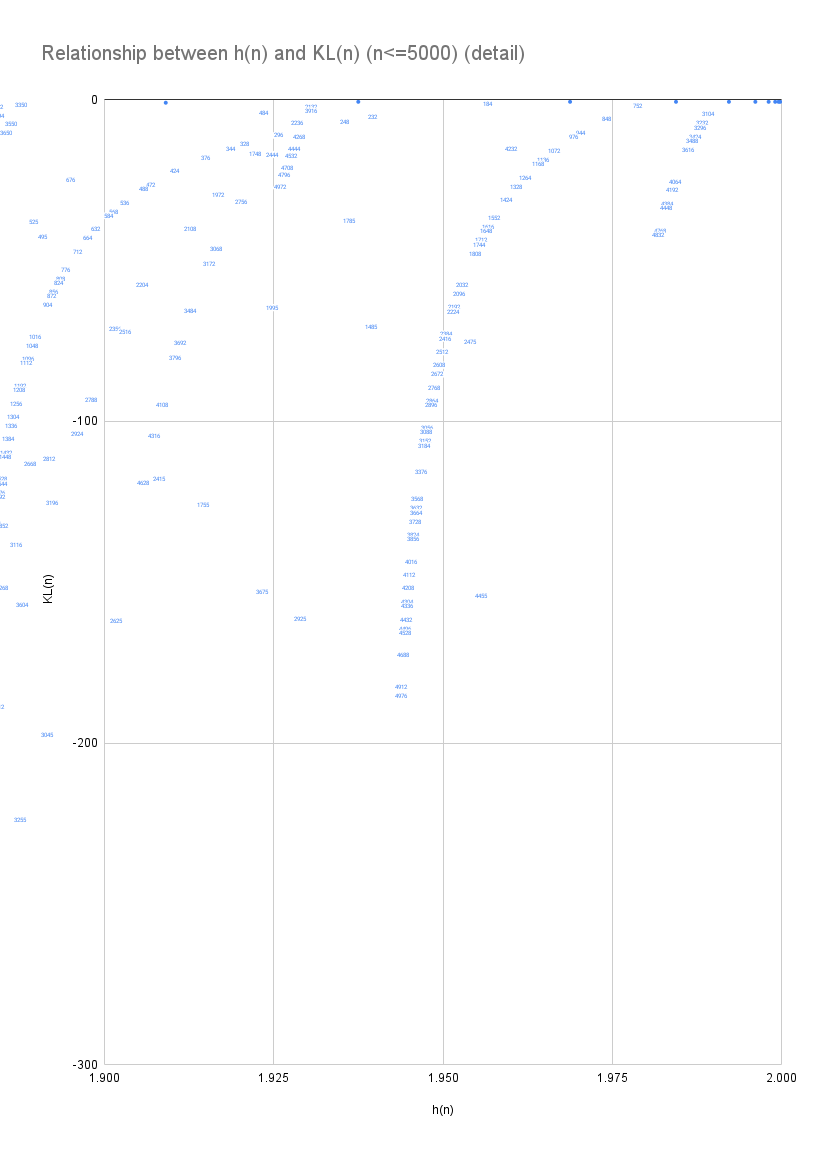}
    \caption{{$h(n)$ and $\kl(n)$ zoomed in further near $h(n)=2$}}
    \label{KL graph 3}
\end{figure}

\FloatBarrier

{\bf Acknowledgments} Dan Gries and Tuesday Mueller-Harder made helpful suggestions.

\end{document}